\documentclass{amsart}
\usepackage{amsmath}
\usepackage{amsfonts}
\usepackage{amssymb}
\usepackage[latin2]{inputenc}
\usepackage{graphicx}
\usepackage{array}
\usepackage{url}
\DeclareMathOperator{\co}{\mathbb C}

\DeclareMathOperator{\lh}{\it L^{\rm2\it}_{h}(\Omega)\rm}
\newtheorem{theorem}{Theorem}[section]

\newtheorem{lemma}[theorem]{Lemma}
\newtheorem{observation}{Observation}
\theoremstyle{definition}

\begin{document}
  \title[completeness of a metric]{On the completeness of a metric related to the Bergman metric}
\author{\.Zywomir Dinew}
\address{Jagiellonian University\newline \indent Faculty of Mathematics and Computer Science\newline \indent
\L ojasiewicza 6\newline \indent 30-348 Krak\'ow\newline \indent Poland  }
\email{Zywomir.Dinew@im.uj.edu.pl}
\thanks{The author was partially supported  by the NCN grant
2011/01/B/ST1/00879}
\subjclass[2010]{Primary 32A36, 32A25; Secondary 32A40, 32F45, 32Q15}
\keywords{ Bergman metric, Ricci curvature, completeness, Kobayashi criterion, hyperconvex domains}
\begin{abstract} We study the completeness of a metric which is related to the Bergman metric of a bounded domain. We provide a criterion for its completeness in the spirit of the Kobayashi criterion for the completeness of the Bergman metric. In particular we prove that in hyperconvex domains our metric is complete. 

\end{abstract}
\maketitle

Recall that in a bounded domain $\Omega\subset\subset\co^{n}$ the Bergman metric is the K\"ahler metric with metric tensor 
\begin{equation}\label{bmdef} T_{i\bar j}(z):=\frac{\partial^{2}}{\partial z_{i}\partial \bar z_{j}}\log K(z,z), z\in\Omega, i,j=1,\dots, n,
\end{equation}
where $K(z,z)$ (or just $K$ for short) is the Bergman kernel (on the diagonal) of the domain $\Omega$.
The length of a vector  $X\in\co^{n}$ ($\cong T_{z}\Omega$) with respect to this metric at $z\in\Omega$ is 
\begin{equation}\label{vlength}
 \beta(z,X)=\beta_{\Omega}(z,X):=\sqrt{\sum_{i,j=1}^{n}T_{i\bar j}(z)X_{i}\bar X_{j}}.
\end{equation}
The Bergman distance between two points $z,\zeta\in\Omega$ is
\begin{equation}\label{bdistance}
dist_{\Omega}(z,\zeta):=\inf_{\gamma \in S}\left\lbrace\int_{0}^{1}\beta(\gamma(t),\gamma'(t))dt \right\rbrace,
 \end{equation}
where $S$ stands for the space of continuous piecewise $\mathcal C^{1}$ and parametrized by the interval $[0,1]$ curves with images in $\Omega$, for which $\gamma(0)=z,\gamma(1)=\zeta$.

 The completeness of the Bergman metric  of $\Omega$ is the property that every Cauchy sequence with respect to $dist_{\Omega}$ has a limit point in $\Omega$ or equivalently, by the Hopf-Rinow theorem, that for any $z\in\Omega, z_{0}\in \partial \Omega$, $$\lim_{\Omega\ni\zeta\to z_{0}}dist_{\Omega}(z,\zeta)=\infty,$$  where the limit is with respect to the Euclidean topology. The completeness of the Bergman metric of bounded domains in $\co^{n}$ has been studied extensively over the years (see \cite{MR0112162}, \cite{MR618233}, \cite{MR1044868},\cite{MR1650305},\cite{MR1704301},\cite{MR1714284},\cite{MR1751175}, \cite{MR1812111}, \cite{MR1800767} for chronological development  and  \cite{MR1314035},\cite{MR2139520} for qualitative results).  
 In this paper we study the completeness of the following (closely related) K\"ahler metric. 
 
\begin{equation}\label{riccimetric}
 \tilde T_{i\bar j}(z):=\left((n+1)T_{i\bar j}(z)+\frac{\partial^{2}}{\partial z_{i}\partial \bar z_{j}}\log \det(T_{p\bar q}(z)_{p,q=1,\dots,n})\right).
\end{equation}
By the well-known formula expressing the Ricci curvature of a K\"ahler metric this can be interpreted as 
\begin{equation}\label{riccimetric1}
 \tilde T_{i\bar j}(z)=(n+1)T_{i\bar j}(z)- Ric_{i\bar j},
\end{equation}
where $Ric_{i\bar j}$ is the Ricci tensor of the Bergman metric. It is well known estimate that the Ricci curvature of the Bergman metric is bounded from above by $n+1$ (see \cite{MR0112162}) and hence it follows that $\tilde T_{i\bar j}$ is positive definite and so it is indeed a metric. This metric enjoys many of the properties of the Bergman metric. In particular it is invariant with respect to biholomorphic mappings. 

 As above, we define
\begin{equation}\label{tildebeta} \tilde\beta(z,X):=\sqrt{\sum_{i,j=1}^{n}\tilde T_{i\bar j}(z)X_{i}\bar X_{j}}\end{equation} and
 \begin{equation}\label{tildedist}\tilde {dist}_{\Omega}(z,\zeta):=\inf_{\gamma \in S}\left\lbrace\int_{0}^{1}\tilde\beta(\gamma(t),\gamma'(t))dt \right\rbrace.\end{equation}  
The  completeness of $\tilde T_{i\bar j}$ is likewise defined as the property that  every Cauchy sequence with respect to $\tilde {dist}_{\Omega}$ has a limit point in $\Omega$ or equivalently that for any $z\in\Omega, z_{0}\in \partial \Omega$, \begin{equation}\label{tildecomplete}\lim_{\Omega\ni\zeta\to z_{0}}\tilde{dist}_{\Omega}(z,\zeta)=\infty.\end{equation}
 Another important property that is shared with the Bergman metric is the fact that domains, which are complete with respect to $\tilde T_{i\bar j}$,  are necessarily pseudoconvex (for the Bergman metric this follows by an old theorem by Bremermann \cite{MR0074058}, for $\tilde T_{i\bar j}$ the proof is virtually the same). For this reason we will restrict our attention to bounded pseudoconvex domains in $\co^{n}$ throughout the paper.

Clearly, if one of the metrics $ T_{i\bar j}, \tilde T_{i\bar j}$ dominates some non-negative multiple the other, then trivially  its completeness follows from the completeness of the dominated metric.  We have

\begin{observation}
 For $\Omega\subset\subset\co^{n}$ 

a) if $ T_{i\bar j}$ is complete and the Ricci curvature of the Bergman metric is bounded above by a constant $C_{1}$, with $C_{1}<n+1$, then $\tilde T_{i\bar j}$ is complete;

b) if $ \tilde T_{i\bar j}$ is complete and the Ricci curvature of the Bergman metric is bounded below by a constant $C_{2}$,  then $ T_{i\bar j}$ is complete.

\end{observation}

In the special case of  a strongly pseudoconvex domain $\Omega$, Fefferman's asymptotic expansion of the Bergman kernel (see \cite{MR0350069}) allows one to compute that the Ricci tensor of the Bergman metric tends to minus identity at the boundary of $\Omega$ (see also \cite{MR1398092}). This, together with the fact that the Bergman metric is complete in strongly pseudoconvex domains, gives one immediately that $\tilde T_{i\bar j}$ is also complete.

Another instance, where the above observation can be used, is when $\Omega$ is a homogeneous bounded domain. Then the Bergman metric is K\"ahler-Einstein and hence it's Ricci curvature is constant. The completeness of $\tilde T_{i\bar j}$ immediately follows.

In general, however, we cannot expect that the conditions on the Ricci curvature of the Bergman metric from the above observation will hold. In fact very few is known about the behavior of the Ricci curvature of the Bergman metric in general bounded domains.    In \cite{MR2640210} and \cite{MR2658116} explicit examples of domains for which both conditions are violated were found. Moreover, such a domain can be hyperconvex (recall that hyperconvex domain is a domain for which there exists a bounded plurisubharmonic (in dimension $1$ subharmonic) exhaustion function). This, together with the fact that bounded hyperconvex domains are complete with respect to the Bergman metric (see \cite{MR1650305} and \cite{MR1714284}), leads one to the following question: whether or not hyperconvex domains are complete with respect to $\tilde T_{i\bar j}$? A more general problem is to study in which classes of weakly pseudoconvex or even non-smooth pseudoconvex domains is $\tilde T_{i\bar j}$ complete.

The examples from \cite{MR2640210} and \cite{MR2658116} enable one to look at the problems studied in this paper form yet another perspective. The completeness of $\tilde T_{i\bar j}$ is equivalent to the completeness of $ T_{i\bar j}-\frac{1}{n+1}Ric_{i\bar j}$, which in certain cases may (presumably) be a gain in the study of the completeness of the Bergman metric.
\begin{section}{Criterions for completeness and statement of the results}

Denote by $\lh:=L^{2}(\Omega)\cap\mathcal O(\Omega)$ the space of square-integrable holomorphic functions. We will benefit from the methods developed to study the completeness of the Bergman metric.
 The main tool for the study of completeness of the Bergman metric is the following criterion due to Kobayashi \cite{MR0112162}, see also \cite{MR0141795}.
\begin{theorem}[Kobayashi]\label{koba}
 Let $\Omega\subset\subset\co^{n}$ be a bounded domain. If for every function $f\in\lh$ and for every boundary point $z_{0}\in\partial\Omega$ and for every sequence $\lbrace z_{s}\rbrace_{s=1}^{\infty}\subset\Omega$ of points in $\Omega$ with limit (in the Euclidean sense) $z_{0}$ there exists a subsequence $\lbrace z_{s_{k}}\rbrace_{k=1}^{\infty}$ such that
\begin{equation}\label{weakkob}\lim_{k\to\infty}\frac{|f(z_{s_{k}})|^2}{K(z_{s_{k}},z_{s_{k}})}=0,\end{equation}
then the Bergman metric of $\Omega$ is complete.
\end{theorem}
This criterion has been modified by several authors (see e.g., \cite{MR2139520}) and a version with weaker assumptions is

\begin{theorem}[B\l ocki]\label{blo}
 Let $\Omega\subset\subset\co^{n}$ be a bounded domain. If for every non-zero $f\in\lh$ and for every boundary point $z_{0}\in\partial\Omega$ and for every sequence $\lbrace z_{s}\rbrace_{s=1}^{\infty}\subset\Omega$ of points in $\Omega$ with limit (in the Euclidean sense) $z_{0}$ there exists a subsequence $\lbrace z_{s_{k}}\rbrace_{k=1}^{\infty}$ such that
$$\lim_{k\to\infty}\frac{|f(z_{s_{k}})|^2}{K(z_{s_{k}},z_{s_{k}})}<\Vert f\Vert_{\lh}^{2},$$
then the Bergman metric of $\Omega$ is complete.
\end{theorem}
We modify the methods of proof of Theorem \ref{blo}  and obtain our

\begin{theorem}\label{main}
 Let $\Omega\subset\subset\co^{n}$ be a bounded domain. If for every $n+1$- tuple of linearly independent $f_{0},f_{1},\dots,f_{n}\in \lh$ and for every boundary point $z_{0}\in\partial\Omega$ and for every sequence $\lbrace z_{s}\rbrace_{s=1}^{\infty}\subset\Omega$ of points in $\Omega$ with limit (in the Euclidean sense) $z_{0}$ there exists a subsequence $\lbrace z_{s_{k}}\rbrace_{k=1}^{\infty}$ such that

\begin{equation}\label{inmain} \left.\begin{array}{c} \lim_{k\to\infty}\frac{\left|\begin{matrix}\det\end{matrix}\begin{pmatrix}
                                                     f_{0}(z)&\dots &f_{n}(z)\\
\frac{ \partial f_{0}}{\partial z_{1}}(z)&\dots & \frac{ \partial f_{n}}{\partial z_{1}}(z)\\
\vdots&\ddots &\vdots\\
\frac{ \partial f_{0}}{\partial z_{n}}(z)&\dots &\frac{ \partial f_{n}}{\partial z_{n}}(z)\\
                                                    \end{pmatrix}
\right|^2 }{\begin{matrix}K^{n+1}\det\left(\frac{\partial^{2}}{\partial z_{i}\partial\bar z_{j}} \log K\right)\end{matrix}}\end{array}\right|_{z=z_{s_{k}}} \end{equation} $$<\det\begin{pmatrix} \langle f_{0},f_{0}\rangle_{\lh} &\cdots&\langle f_{n},f_{0}\rangle_{\lh}\\
 \vdots&\ddots&\vdots\\
  \langle f_{0},f_{n}\rangle_{\lh}&\cdots&\langle f_{n},f_{n}\rangle_{\lh} \end{pmatrix},$$
then $\tilde T_{i\bar j}$  is complete.
\end{theorem}
Note that the right hand side of the above expression is the Gramian of the vectors $f_{0},f_{1},\dots,f_{n} $, which is positive, and hence a stronger assumption, which would also imply the completeness, is to require the limit in (\ref{inmain}) to be $0$.

To obtain this, we modify a construction of Lu Qi-Keng (see \cite{MR2447420}), which goes as follows. If $\varphi_{0},\varphi_{1},\dots$ is a orthonormal basis of $\lh$, then one can embed holomorphically the domain $\Omega$ into the 
 infinite dimensional Grassmannian of $n$- dimensional subspaces of $\ell^{2}$, denoted by $\mathbb F(n,\infty)$, by means of
\begin{equation*}\Omega\ni z\to\left[\left.\left(\begin {smallmatrix}
                        \varphi_{0}\frac{\partial\varphi_{1}}{\partial z_{1}}
-\varphi_{1}\frac{\partial\varphi_{0}}{\partial z_{1}} & \varphi_{0}\frac{\partial\varphi_{2}}{\partial z_{1}}-\varphi_{2}\frac{\partial\varphi_{0}}{\partial z_{1}}&\varphi_{1}\frac{\partial\varphi_{2}}{\partial z_{1}}-\varphi_{2}\frac{\partial\varphi_{1}}{\partial z_{1}} & \cdots\\
\vdots&\vdots&\vdots &\cdots\\
  \varphi_{0}\frac{\partial\varphi_{1}}{\partial z_{n}}
-\varphi_{1}\frac{\partial\varphi_{0}}{\partial z_{n}} &
\varphi_{0}\frac{\partial\varphi_{2}}{\partial
z_{n}}-\varphi_{2}\frac{\partial\varphi_{0}}{\partial
z_{n}}&\varphi_{1}\frac{\partial\varphi_{2}}{\partial
z_{n}}-\varphi_{2}\frac{\partial\varphi_{1}}{\partial z_{n}} &
\cdots                   \end {smallmatrix}\right)\right|_{z} \right]\in \mathbb
F(n,\infty),
\end{equation*}
where $[\cdot]$ is the equivalence relation  between $n$- dimensional subspaces of $\ell^{2}$ defining the points in the Grassmannian. This Grassmannian can further be embedded into some projective space by means of the Pl\"ucker embedding and eventually the pullback  of the Fubini-Study metric by the composition of these two embeddings is exactly $ \tilde T_{i\bar j}$ (see \cite{MR2817571}). This approach has some significant disadvantages. The embedding is not independent of the basis, but the main problem is that, because partial derivatives of $L^2$ functions need not be $L^2$, the Grassmannian consists of subspaces of $\ell^2$ and not  $\lh$. Intuitively this is like a pointwise construction which due to the lack of uniformity is not enough to obtain our goals. Our new construction is also far simpler.

 With the help of Theorem \ref{main} we prove.
\begin{theorem}\label{hyperc}
 Bounded hyperconvex domains are complete with respect to $\tilde T_{i\bar j}$.
\end{theorem}

In particular all pseudoconvex domains with Lipschitz boundaries, which are known to be hyperconvex (see \cite{MR881709}), are complete with respect to $\tilde T_{i\bar j}$.

\end{section}

\begin{section}{Exterior products of Hilbert spaces}
We begin with some basic facts about Hilbert spaces, which are not commonly seen in the theory of Bergman spaces.   Let $V$ be a complex vector space. We define the (algebraic) tensor product vector space $V\otimes V$ as the quotient vector space $^U/_W$ of some vector spaces $U$ and $W$. Here $U$ is the vector space generated by  all pairs $(\alpha,\beta)\in V\times V$ as finite formal linear combinations with complex coefficients and $W$ is the space generated in the same way by all elements of the following types $$(\alpha+\beta,\gamma)-(\alpha,\gamma)-(\beta,\gamma);$$$$ (\alpha,\beta+\gamma)-(\alpha,\beta)-(\alpha,\gamma);$$$$ (a\alpha,\beta)-a(\alpha,\beta);$$$$ (\alpha,a\beta)-a(\alpha,\beta),$$
where $\alpha,\beta,\gamma\in V, a\in\co $. Clearly $W$ is a subspace of $U$. The tensor product $\alpha\otimes \beta$, which is a equivalence class, can be interpreted as the affine space $(\alpha,\beta)+W$. Now the wedge (or exterior) product $V\wedge V$ is defined as the quotient vector space $^{V\otimes V}/_S$, where $S\subset V\otimes V$ is the vector space generated by all elements of the type $\alpha\otimes \alpha$, where $\alpha\in V$. Again $\alpha\wedge\beta$ is a equivalence class, which can be interpreted as the affine space $\alpha\otimes\beta+S\subset V\otimes V$.

Let $H$ be a separable Hilbert space, carrying the inner product $\langle\cdot,\cdot\rangle_{H}$. Now $H\wedge H$ makes sense at least as a vector space. This vector space $H\wedge H$ consists of all finite sums of the type $\sum_{i=1}^{m}a_{i}\alpha_{i}\wedge\beta_{i}$, where $a_{i}\in\co, \alpha_{i},\beta_{i}\in H,m\in \mathbb N$. We endow this space with a inner product defined as follows. For elements of the type $\alpha\wedge\beta$ and $\gamma\wedge\delta$, where $\alpha,\beta,\gamma,\delta\in H$

\begin{equation}\label{inner2}\langle \alpha\wedge\beta,\gamma\wedge\delta \rangle_{H\wedge H}:=\det\begin{pmatrix} \langle\alpha,\gamma\rangle_{H}&\langle\alpha,\delta\rangle_{H}\\
   \langle\beta,\gamma\rangle_{H}&\langle\beta,\delta\rangle_{H}                                                                                        \end{pmatrix}.\end{equation}
After defining the inner product on such vectors we extend it on the whole vector space $H\wedge H$ by linearity. Now we perform the completion of $H\wedge H$ with respect to $\langle \cdot,\cdot\rangle_{H\wedge H}$, that is we allow not only finite but also countable combinations  $\sum_{i=1}^{\infty}a_{i}\alpha_{i}\wedge\beta_{i}$, obeying the natural restriction that $\sum_{i=1}^{\infty}|a_{i}|^2<\infty$. By abusing notation, we agree to call this completion also $H\wedge H$. The inner product also extends to the completed vector space and again by abusing notation we call the extension $\langle\cdot,\cdot\rangle_{H\wedge H}$. Now it is easy to see that $(H\wedge H, \langle \cdot,\cdot\rangle_{H\wedge H})$ is a Hilbert space. It is also easy to see that this Hilbert space is separable.

Likewise if we take $n+1$ copies of a Hilbert space $F$, we can define the Hilbert space $(F\wedge\cdots\wedge F, \langle \cdot,\cdot\rangle_{F\wedge\cdots\wedge F})$, which is the completion of the vector space  $F\wedge\cdots\wedge F$ with respect to the inner product, which is the linear extension of

\begin{equation}\label{innern}\langle \alpha_{0}\wedge\cdots\wedge\alpha_{n},\beta_{0}\wedge\cdots\wedge\beta_{n}\rangle_{F\wedge\cdots\wedge F}:=\det\begin{pmatrix} \langle\alpha_{0},\beta_{0}\rangle_{F}&\cdots&\langle\alpha_{0},\beta_{n}\rangle_{F}\\
 \vdots&\ddots&\vdots\\
  \langle\alpha_{n},\beta_{0}\rangle_{F}&\cdots&\langle\alpha_{n},\beta_{n}\rangle_{F} \end{pmatrix}.    \end{equation}
It is a matter of algebraic manipulations to see  that the continuous dual space of $F\wedge\cdots\wedge F$ satisfies
\begin{equation}\label{dual}\left(F\wedge\cdots\wedge F\right)'=F'\wedge\cdots\wedge F'.\end{equation}
A proof of this fact can be found in \cite{MR633754}.

A element $\alpha\in F\wedge\cdots\wedge F$ which can be represented as $\alpha=\alpha_{0}\wedge_{ }\alpha_{1}\wedge_{}\cdots\wedge_{}\alpha_{n}$, for some $ \alpha_{i}\in F, i=0,\dots,n$ will be called decomposable (the terms pure, monomial, simple and completely reducible are also frequent in the literature). Clearly not all elements of $F\wedge\cdots\wedge F$ are decomposable.
 There is a criterion for determining whether a non-zero vector is decomposable or not, known as Pl\"ucker (or Pl\" ucker-Grassmann) conditions. To introduce it we need more notation. Let $J$ be a $s$- tuple of natural numbers $j_{1}<\dots< j_{s}$. We denote by $e_{J}$ the vector $e_{j_{1}}\wedge\dots\wedge e_{j_{s}}$, where $e_{j}$ is a fixed orthonormal basis of a separable Hilbert space $E$. Clearly $e_{J}\in E\wedge\dots\wedge E$, where the exterior product is taken $s$ times, and moreover the vectors $e_{J}$, for all possible $s$-tuples $J$ of distinct natural numbers, form a orthonormal basis of $E\wedge\dots\wedge E$. We can therefore expand a vector $\alpha \in E\wedge\dots\wedge E$ as $\alpha=\sum_{J}a_{J}e_{J}$, where $a_{J}=\langle \alpha,e_{J}\rangle_{E\wedge\dots\wedge E}\in\co$. Now a non-zero vector $\alpha$ is decomposable if and only if for all $I\subset \mathbb N^{s-1}$ and for all $L\subset \mathbb N^{s+1}$, both $J$ and $L$ without recurring elements, such that $I\cap L=\emptyset$, the following equality holds

\begin{equation}\label{plugra} \sum _{i\in L}\rho_{J,L,i}a_{I\cup \{i\}}a_{L\setminus \{i\}}=0, 
\end{equation}
where $\rho_{J,L,i}=1$ if $\sharp \{ j\in L: j< i\}\equiv \sharp \{ j\in I: j< i\}  (\mod 2)$ and $\rho_{J,L,i}=-1$ otherwise. Also in the index notation $a_{I\cup \{i\}}$ (respectively $a_{L\setminus \{i\}}$) it should be clarified that the elements of the sets $I\cup \{i\}$ (respectively $L\setminus \{i\}$) are ordered in a increasing fashion. For a proof see \cite{GALLIER}, Chapter 22. Actually in \cite{GALLIER} only the finite-dimensional case is considered, however, one should take the continuous dual space instead of the algebraic dual space and the argument goes mutatis-mutandis.

\begin{lemma}\label{sequence} If a sequence $\lbrace \alpha_{i}\rbrace_{i=1}^{\infty}$ of unit vectors in $F\wedge\cdots\wedge F$ has a limit $\alpha\in F\wedge\cdots\wedge F$ in the norm topology and moreover each $\alpha_{i}$ is of the form $b_{i}\alpha_{i0}\wedge_{ }\alpha_{i1}\wedge\cdots\wedge\alpha_{in}$, where $b_{i}\in\co,\alpha_{ij} \in\lh', j=0,\dots,n, i=1,\dots$, then also $\alpha$ is a unit vector of the form
$b\alpha_{0}\wedge_{ }\alpha_{1}\wedge_{}\cdots\wedge\alpha_{n}$, for some $b\in\co,\alpha_{j} \in\lh',j=0,\dots,n$ (that is the limit is a decomposable vector). 
\end{lemma}
First observe that it is not true in general that if a sequence $f_{s}\wedge g_{s}$, $f_{s},g_{s}\in E$, for some Hilbert space $E$, has a limit in $E\wedge E$ then necessarily $f_{s}$ and $g_{s}$ both have limits in $E$ and the simplest counterexample is just $f_{s}=sf$, $g_{s}=\frac{1}{s}g$ , for some fixed $ f,g\in E$. 
\begin{proof} We expand the sequence elements, as well as the limit, into
$$\alpha_{i}=\sum_{J}a_{J}^{i}e_{J}, \alpha= \sum_{J}a_{J}e_{J}.$$
Since $\Vert \alpha_{i}-\alpha\Vert_{F\wedge\cdots\wedge F}\to 0$, it follows that $|a_{J}^{i}-a_{J}|\to 0$.  By the assumption and the Pl\"ucker relations (\ref{plugra}) we have
$$\sum _{i\in L}\rho_{J,L,i}a_{I\cup \{i\}}^{i}a_{L\setminus \{i\}}^{i}=0,$$
for all subsets $J\subset\mathbb N^{n},L\subset\mathbb N^{n+2}$, without repetitions, such that $J\cap L=\emptyset$. Now it is obvious that also
$$\sum _{i\in L}\rho_{J,L,i}a_{I\cup \{i\}}a_{L\setminus \{i\}}=0.$$

\end{proof}

For more on these items one should consult \cite{MR633754}, Chapter $5, \S 3,4$, where tensor and exterior products of Hilbert spaces are explicitly considered, \cite{MR0274237}, Chapter $3$, for more results but in a more abstract algebraic setting and also \cite{GALLIER}, Chapter $22$, where the concepts of decomposable vectors and tests for decomposability are very clearly presented, however, only in finite dimensions.
\begin{section}{the construction}
In our case $F$ will be $\lh'$ - the Hilbert space which is the continuous dual space of $\lh$ and so $F\wedge\cdots\wedge F=\lh'\wedge\cdots\wedge\lh'$ can be identified with the Hilbert space of multilinear antisymmetric continuous mappings (forms) from $\lh\times\cdots\times\lh$ to $\co$. Actually the forms are defined on $\lh\wedge\cdots\wedge\lh$ rather than on $\lh\times\cdots\times\lh$ but the definition can be extended in a obvious and canonical way.  A element $\alpha$ of the Hilbert space $F\wedge\cdots\wedge F$ can be written down as a linear combination of the form
 $$\alpha=\sum_{i=1}^{\infty} a_{i}\alpha_{i0}\wedge_{ }\alpha_{i1}\wedge_{}\cdots\wedge_{}\alpha_{in},$$
where $\{a_{i}\}_{i=1}^{\infty}\in\ell^{2}, \alpha_{ij}\in \lh',j=0,\dots,n,i=1,\dots$. The aforementioned identification with a multilinear antisymmetric form is realized by first identifying elements of the type $\alpha_{i0}\wedge_{ }\alpha_{i1}\wedge_{}\cdots\wedge_{}\alpha_{in}$ by
$$\alpha_{i0}\wedge_{ }\alpha_{i1}\wedge_{}\cdots\wedge_{}\alpha_{in}\cong$$
\begin{equation}\label{multilinear}
\lh\times\cdots\times\lh\ni(f_{0},\dots,f_{n})\to \det\begin{pmatrix} \alpha_{i0}(f_{0})&\cdots&\alpha_{i0}(f_{n})\\
 \vdots&\ddots&\vdots\\
  \alpha_{in}(f_{0})&\cdots&\alpha_{in}(f_{n}) \end{pmatrix} \end{equation}

$$= \alpha_{i0}\wedge_{ }\alpha_{i1}\wedge_{}\cdots\wedge_{}\alpha_{in}(f_{0},\dots,f_{n})\in\co$$
and extending it linearly on the whole $\lh'\wedge\cdots\wedge\lh'$ afterwards. This is consistent with the introduced inner product and hence the correspondence is clearly a isomorphism of Hilbert spaces.

By the Cauchy estimates the following linear mappings are continuous

\begin{equation}\label{irepres}i(z):\lh\ni f\to f(z)\in\co,\end{equation}
\begin{equation}\label{j1repres}j_{1}(z):\lh\ni f\to \frac{\partial f}{\partial z_{1}}(z)\in\co,\end{equation}
$$\dots$$
\begin{equation}\label{jnrepres}j_{n}(z):\lh\ni f\to \frac{\partial f}{\partial z_{n}}(z)\in\co.\end{equation}
By the Riesz theorem for every $l\in\lh'$ there is a unique $l'\in \lh$ such that $ l(\cdot)=\langle \cdot, l' \rangle_{\lh}$. Moreover, $ \langle k, l \rangle_{\lh'}=\overline{\langle k', l' \rangle}_{\lh}=\langle    l',k' \rangle_{\lh}$.
In our case one can easily check by using the reproducing property of the Bergman kernel that
\begin{equation}\label{irep}i(z)'=K(\cdot,z)\in\lh,\end{equation}
\begin{equation}\label{jsrep}j_{s}(z)'=\left.\frac{\partial K(\cdot,\zeta)}{\partial \bar \zeta_{s}}\right|_{\zeta=z}\in\lh, s=1,\dots,n.\end{equation}

Let $\mathbb P(F\wedge\cdots\wedge F)$ be the projectivization of the Hilbert space $F\wedge\cdots\wedge F$, that is the quotient space $^{F\wedge\cdots\wedge F}/_{\sim}$ with respect to the following  (projective) equivalence relation. For $u,v\in 
F\wedge\cdots\wedge F$ we have $u\sim v$ if and only if $u=cv$, for some $c\in\co\setminus\{0\}$. For more on projectivizations of infinite dimensional Hilbert spaces see \cite{MR0112162}.
We embed  $\Omega$ into the projective space $\mathbb P(F\wedge\cdots\wedge F)$
by the mapping

\begin{equation}\label{embed}\Omega \ni z\to[i(z)\wedge j_{1}(z)\wedge j_{2}(z)\wedge \cdots\wedge  j_{n}(z)]\in \mathbb P(F\wedge\cdots\wedge F),\end{equation}
where $[.]$ is the equivalence class with respect to $\sim$. This is a holomorphic embedding. 
\end{section}

\begin{lemma}\label{value} The value of $i(z)\wedge j_{1}(z)\wedge j_{2}(z)\wedge \cdots\wedge  j_{n}(z)$, interpreted as an antisymmetric  multilinear form on $\lh\times\cdots\times\lh$, at the point $(f_{0},f_{1},\dots,f_{n})\in \lh\times\dots\times \lh$ is
 $$\det\begin{pmatrix}
                                                     f_{0}(z)&\dots &f_{n}(z)\\
\frac{ \partial f_{0}}{\partial z_{1}}(z)&\dots & \frac{ \partial f_{n}}{\partial z_{1}}(z)\\
\vdots&\ddots &\vdots\\
\frac{ \partial f_{0}}{\partial z_{n}}(z)&\dots &\frac{ \partial f_{n}}{\partial z_{n}}(z)\\
                                                    \end{pmatrix}.$$
\end{lemma}

 The proof is a immediate consequence of (\ref{multilinear}), (\ref{irepres}), (\ref{j1repres}) and (\ref{jnrepres}).

\begin{lemma}\label{norm} The square of the norm of $i(z)\wedge j_{1}(z)\wedge \cdots\wedge j_{n}(z)$ in $F\wedge\cdots\wedge F$ equals

$$\left\Vert i(z)\wedge j_{1}(z)\wedge \cdots\wedge  j_{n}(z)\right\Vert_{F\wedge\cdots\wedge F}^{2}=\left. K^{n+1}\det\left(\frac{\partial^{2}}{\partial z_{i}\partial\bar z_{j}} \log K\right)\right|_{z}.$$
 
\end{lemma}
\begin{proof}
 By (\ref{innern}), the Riesz theorem, (\ref{irep}), (\ref{jsrep}) and the reproducing property of the Bergman kernel we have
$$\left\Vert i(z)\wedge j_{1}(z)\wedge \cdots\wedge j_{n}(z)\right\Vert_{F\wedge\cdots\wedge F}^{2}$$
$$=\Big\langle i(z)\wedge j_{1}(z)\wedge \cdots\wedge  j_{n}(z), i(z)\wedge j_{1}(z)\wedge \cdots\wedge  j_{n}(z)\Big\rangle_{F\wedge\cdots\wedge F}$$
$$=\det\begin{pmatrix}
        \langle i(z), i(z)\rangle_{\lh'} & \langle i(z), j_{1}(z)\rangle_{\lh'}&\dots &\langle i(z),  j_{n}(z)\rangle_{\lh'}\\
\langle j_{1}(z), i(z)\rangle_{\lh'} & \langle j_{1}(z), j_{1}(z)\rangle_{\lh'}&\dots &\langle j_{1}(z),  j_{n}(z)\rangle_{\lh'}\\
\vdots &\vdots&\ddots & \vdots \\
\langle  j_{n}(z), i(z)\rangle_{\lh'} &\langle  j_{n}(z), j_{1}(z)\rangle_{\lh'}&\dots &\langle j_{n}(z),  j_{n}(z)\rangle_{\lh'}
       \end{pmatrix}
$$

$$=\left.\det\begin{pmatrix}
        \langle K, K\rangle_{\lh} & \langle   \frac{\partial K}{\partial \bar \zeta_{1}},K\rangle_{\lh}&\dots &\langle  \frac{\partial K}{\partial \bar \zeta_{n}},K\rangle_{\lh}\\
\langle K,\frac{\partial K}{\partial \bar \zeta_{1}}\rangle_{\lh} & \langle \frac{\partial K}{\partial \bar \zeta_{1}}, \frac{\partial K}{\partial \bar \zeta_{1}}\rangle_{\lh}&\dots &\langle  \frac{\partial K}{\partial \bar \zeta_{n}},\frac{\partial K}{\partial \bar \zeta_{1}}\rangle_{\lh}\\
\vdots &\vdots&\ddots & \vdots \\
\langle K, \frac{\partial K}{\partial \bar \zeta_{n}}\rangle_{\lh} &\langle    \frac{\partial K}{\partial \bar \zeta_{1}},\frac{\partial K}{\partial \bar \zeta_{n}}\rangle_{\lh}&\dots &\langle \frac{\partial K}{\partial \bar \zeta_{n}},  \frac{\partial K}{\partial \bar \zeta_{n}}\rangle_{\lh}
       \end{pmatrix}\right|_{z}
$$

$$=\left.\det\begin{pmatrix}
        K &   \frac{\partial K}{\partial \bar \zeta_{1}}&\dots & \frac{\partial K}{\partial  \bar\zeta_{n}}\\
 \frac{\partial K}{\partial  \zeta_{1}} &  \frac{\partial^2 K}{\partial \zeta_{1}\partial \bar \zeta_{1}}, &\dots &\frac{\partial^2 K}{\partial \zeta_{1}\partial \bar \zeta_{n}}\\
\vdots &\vdots&\ddots & \vdots \\
 \frac{\partial K}{\partial  \zeta_{n}} & \frac{\partial^2 K}{\partial \zeta_{n}\partial \bar \zeta_{1}}, &\dots &   \frac{\partial^2 K}{\partial  \zeta_{n}\partial\bar \zeta_{n}}
       \end{pmatrix}\right|_{z}.$$
By a well known formula (see e.g., \cite{MR1398092}) the last expression equals  

$$\left. K^{n+1}\det\left(\frac{\partial^{2}}{\partial z_{i}\partial\bar z_{j}} \log K\right)\right|_{z}.$$
 \end{proof}

\begin{theorem}\label{isometry}
 The embedding (\ref{embed}) is isometric, that is, the pullback of the Fubini-Study metric on $\mathbb P(F\wedge\cdots\wedge F)$ is exactly the metric $\tilde T_{i\bar j}$.  
\end{theorem}
\begin{proof}
 First recall that the Fubini-Study metric on a projectivization $\mathbb P(E)$ of a Hilbert space $E$ at the point $[\zeta]\in \mathbb P(E)$ has the following metric tensor
$$\mathcal FS_{p\bar q}:=\frac{\partial^{2}}{\partial \zeta_{p}\partial \bar \zeta_{q}}\log \Vert \zeta\Vert_{E}^{2}.$$
Note that the definition does not depend ot the choice of the (non-zero) representative $\zeta \in [\zeta]$. Let the image of the point $z\in\Omega$ be $[i(z)\wedge j_{1}(z)\wedge j_{2}(z)\wedge \cdots\wedge  j_{n}(z)]=[\zeta]$. By Lemma (\ref{norm}) the pullback of the Fubini-Study metric is the metric with metric tensor

$$[i(\cdot)\wedge j_{1}(\cdot)\wedge j_{2}(\cdot)\wedge \cdots\wedge  j_{n}(\cdot)]^{\ast}(\mathcal FS_{p\bar q})$$$$=\frac{\partial^{2}}{\partial z_{i}\partial \bar z_{j}}\log\left(\left. K^{n+1}\det\left(\frac{\partial^{2}}{\partial z_{r}\partial\bar z_{s}} \log K\right)\right|_{z}\right)$$
$$=(n+1)T_{i\bar j}(z)+\frac{\partial^{2}}{\partial z_{i}\partial \bar z_{j}}\log \Big(\det T_{r\bar s}(z)_{r,s=1,\dots,n}\Big)=\tilde T_{i\bar j}$$
\end{proof}
For more on the Fubini-Study metric on projectivizations of Hilbert spaces and related items see \cite{MR0112162}.
\begin{theorem}\label{fraction}
 The following equality holds

$$\left. K^{n+1}\det\left(\frac{\partial^{2}}{\partial z_{i}\partial\bar z_{j}} \log K\right)\right|_{z}$$$$=\sup_{\substack{(f_{0},\dots,f_{n})\in\lh\times\dots\times\lh:\\ f_{0}\wedge\dots\wedge f_{n}\neq 0}}\frac{\left|\det\begin{pmatrix}
                                                     f_{0}(z)&\dots &f_{n}(z)\\
\frac{ \partial f_{0}}{\partial z_{1}}(z)&\dots & \frac{ \partial f_{n}}{\partial z_{1}}(z)\\
\vdots&\ddots &\vdots\\
\frac{ \partial f_{0}}{\partial z_{n}}(z)&\dots &\frac{ \partial f_{n}}{\partial z_{n}}(z)\\
                                                    \end{pmatrix}
\right|^2}{\det\begin{pmatrix} \langle f_{0},f_{0}\rangle_{\lh} &\cdots&\langle f_{n},f_{0}\rangle_{\lh}\\
 \vdots&\ddots&\vdots\\
  \langle f_{0},f_{n}\rangle_{\lh}&\cdots&\langle f_{n},f_{n}\rangle_{\lh} \end{pmatrix}}.$$
\end{theorem}

\begin{proof}

We have the isometry
$$\left(\lh\wedge\cdots\wedge\lh\right)^{'}\cong\lh'\wedge\cdots\wedge\lh'.$$
As usual the  norm of a linear functional is $$\Vert \alpha \Vert_{\left(\lh\wedge\cdots\wedge\lh\right)^{'}}=\sup_{}\frac{|\alpha(f)|}{\Vert f\Vert_{\lh\wedge\cdots\wedge\lh}},$$
where the supremum is taken over all nonzero $f\in \lh\wedge\cdots\wedge\lh$. By the Riesz theorem the supremum is achieved at the vector $f=\alpha'$. When $\alpha$ is decomposable, we use the fact that $\alpha'$ is also decomposable.  In fact if $\alpha=\alpha_{0}\wedge\dots\wedge\alpha_{n}$ then $\alpha'=\alpha_{0}'\wedge\dots\wedge\alpha_{n}'$.
By Lemma \ref{norm}, the decomposability of $i(z)\wedge j_{1}(z)\wedge \cdots\wedge  j_{n}(z)$ and Lemma \ref{value} we have
 $$\left. K^{n+1}\det\left(\frac{\partial^{2}}{\partial z_{i}\partial\bar z_{j}} \log K\right)\right|_{z}=\left\Vert i(z)\wedge j_{1}(z)\wedge \cdots\wedge  j_{n}(z)\right\Vert_{\lh'\wedge\cdots\wedge \lh'}^{2}$$
$$=\left\Vert i(z)\wedge j_{1}(z)\wedge \cdots\wedge  j_{n}(z)\right\Vert_{\left(\lh\wedge\cdots\wedge\lh\right)'}^{2}$$$$=\sup_{0\neq f_{0}\wedge\dots\wedge f_{n}\in \lh\wedge\cdots\wedge\lh}\frac{\left|i(z)\wedge j_{1}(z)\wedge \cdots\wedge  j_{n}(z)(f_{0},\dots,f_{n})\right|^{2}}{\Vert f_{0}\wedge\dots\wedge f_{n}\Vert_{\lh\wedge\cdots\wedge\lh}^{2}}$$

$$=\sup_{\substack{(f_{0},\dots,f_{n})\in\lh\times\dots\times\lh:\\ f_{0}\wedge\dots\wedge f_{n}\neq 0}}\frac{\left|\det\begin{pmatrix}
                                                     f_{0}(z)&\dots &f_{n}(z)\\
\frac{ \partial f_{0}}{\partial z_{1}}(z)&\dots & \frac{ \partial f_{n}}{\partial z_{1}}(z)\\
\vdots&\ddots &\vdots\\
\frac{ \partial f_{0}}{\partial z_{n}}(z)&\dots &\frac{ \partial f_{n}}{\partial z_{n}}(z)\\
                                                    \end{pmatrix}
\right|^2}{\det\begin{pmatrix} \langle f_{0},f_{0}\rangle_{\lh} &\cdots&\langle f_{n},f_{0}\rangle_{\lh}\\
 \vdots&\ddots&\vdots\\
  \langle f_{0},f_{n}\rangle_{\lh}&\cdots&\langle f_{n},f_{n}\rangle_{\lh} \end{pmatrix}}.$$
\end{proof}

\end{section}
\begin{section}{Proofs of the Theorems and open problems}
\begin{proof}[Proof of Theorem \ref{main}] We proceed as in \cite{MR0112162}.
 Suppose that the metric $\tilde T_{i\bar j}$ is not complete in $\Omega$. We choose a Cauchy (with respect to $\tilde T_{i\bar j}$) sequence $\lbrace z_{s}\rbrace_{s=1}^{\infty}\subset\Omega$  which has no convergent (again with respect to $\tilde T_{i\bar j}$) subsequence. Now we use Theorem \ref{isometry}  and embed holomorphically and isometrically $\Omega$ with the metric $\tilde T_{i\bar j}$ into $\mathbb P(F\wedge\cdots\wedge F)$ with the Fubini-Study metric by the mapping (\ref{embed}).  The image sequence $[i(z_{s})\wedge j_{1}(z_{s})\wedge \cdots\wedge  j_{n}(z_{s})]$ is also a Cauchy sequence with respect to the Fubini-Study metric, because isometries do not increase distance. The space $\mathbb P(F\wedge\cdots\wedge F)$ is, however, complete and hence the image sequence has a convergent subsequence (with respect to the Fubini-Study metric)  $[i(z_{s_{k}})\wedge j_{1}(z_{s_{k}})\wedge \cdots\wedge  j_{n}(z_{s_{k}})]$ with limit $f\in\mathbb P(F\wedge\cdots\wedge F)$. This means that also the unit vectors  
$$e^{i\theta_{k}}\frac{i(z_{s_{k}})\wedge j_{1}(z_{s_{k}})\wedge \cdots\wedge  j_{n}(z_{s_{k}})}{\Vert i(z_{s_{k}})\wedge j_{1}(z_{s_{k}})\wedge \cdots\wedge  j_{n}(z_{s_{k}})\Vert}\in F\wedge\cdots\wedge F ,$$
which represent the above classes, converge for a proper choice of $\theta_k\in[0,2\pi)$ in $F\wedge\cdots\wedge F$ to some $\alpha$, which represents the class $f$. Now  $\alpha$ is a unit vector and moreover by Lemma \ref{sequence} $\alpha=\alpha_{0}\wedge\dots\wedge\alpha_{n}$, for some $\alpha_{0},\dots,\alpha_{n}\in \lh'$. The vector $\alpha$ is nonzero and hence the $\alpha_{s}$s are linearly independent in $\lh'$. For each $\alpha_s$ take the Hilbert dual  $f_{s}\in\lh$. Clearly also the $f_{s}$s are linearly independent.  Now by Lemmas \ref{value}, \ref{norm} and the formula (\ref{multilinear})

$$\left.\begin{array}{c} \frac{\left|\begin{matrix}\det\end{matrix}\begin{pmatrix}
                                                     f_{0}(z)&\dots &f_{n}(z)\\
\frac{ \partial f_{0}}{\partial z_{1}}(z)&\dots & \frac{ \partial f_{n}}{\partial z_{1}}(z)\\
\vdots&\ddots &\vdots\\
\frac{ \partial f_{0}}{\partial z_{n}}(z)&\dots &\frac{ \partial f_{n}}{\partial z_{n}}(z)\\
                                                    \end{pmatrix}
\right|^2 }{\begin{matrix}K^{n+1}\det\left(\frac{\partial^{2}}{\partial z_{i}\partial\bar z_{j}} \log K\right)\end{matrix}}\end{array}\right|_{z=z_{s_{k}}}$$
$$=\frac{\left|i(z_{s_{k}})\wedge j_{1}(z_{s_{k}})\wedge \cdots\wedge  j_{n}(z_{s_{k}})(f_{0},\dots,f_{n})\right|^2}{\Vert i(z_{s_{k}})\wedge j_{1}(z_{s_{k}})\wedge \cdots\wedge  j_{n}(z_{s_{k}})\Vert^2} $$
$$\to |\alpha_{0}\wedge\dots\wedge\alpha_{n} (f_{0},\dots,f_{n})|^2=\det\begin{pmatrix} \langle f_{0},f_{0}\rangle_{\lh} &\cdots&\langle f_{n},f_{0}\rangle_{\lh}\\
 \vdots&\ddots&\vdots\\
  \langle f_{0},f_{n}\rangle_{\lh}&\cdots&\langle f_{n},f_{n}\rangle_{\lh} \end{pmatrix}.$$
This contradicts the assumptions of Theorem \ref{main}.
 
\end{proof}
Recall that the pluricomplex  Green function of the bounded domain $\Omega\subset\co^{n}$ with logarithmic singularity at $z\in\Omega$ is the function
$$G_{\Omega}(\cdot, z):=\sup_{\varphi\in PSH(\Omega)}\lbrace \varphi(\cdot): \varphi<0,  \limsup_{\zeta\to z}( f(\zeta)-\log|\zeta-z|)<\infty\rbrace,$$
where $PSH(\Omega)$ is the space of plurisubharmonic functions on $\Omega$.
The function $G_{\Omega}(\cdot, z)$ is plurisubharmonic and negative in $\Omega$. We will need a lemma.
\begin{lemma}\label{estimate}
 For every bounded pseudoconvex domain $\Omega$ there exists a constant $C>0$ such that for every $f\in\lh$ one can find $\tilde f\in \lh$ such that for given $z\in\Omega$, $f(z)=\tilde f(z)$ and $\frac{\partial f}{\partial z_j}(z)= \frac{\partial \tilde f}{\partial z_j}(z)$ and moreover $$\int_{\Omega}|\tilde f|^2 d\lambda\leq C\int_{\{G_{\Omega}(\cdot,z)<-1\}}|f|^2d \lambda,$$
where $d\lambda$ is the Lebesgue measure.
\end{lemma}
This is a simpler version of  Lemma $4.2$ in \cite{MR1799743}. The proof uses H\"ormander's estimates for the $\bar\partial$- equation and can be found in \cite{MR1799743}. The constant $C$ can be chosen to be $1+e^{4n+7+(max_{\Omega}|z|)^2}$.

\begin{proof}[Proof of Theorem \ref{hyperc}]

  At the point $z_{s_{k}}$ we construct the corresponding functions $\tilde f_{j}$ for each of the functions $f_{j},j=0,\dots, n$ from Lemma \ref{estimate}. Then by Theorem \ref{fraction}, Hadamard's inequality and Lemma \ref{estimate} we have 
$$\left.\begin{array}{c} \frac{\left|\begin{matrix}\det\end{matrix}\begin{pmatrix}
                                                     f_{0}(z)&\dots &f_{n}(z)\\
\frac{ \partial f_{0}}{\partial z_{1}}(z)&\dots & \frac{ \partial f_{n}}{\partial z_{1}}(z)\\
\vdots&\ddots &\vdots\\
\frac{ \partial f_{0}}{\partial z_{n}}(z)&\dots &\frac{ \partial f_{n}}{\partial z_{n}}(z)\\
                                                    \end{pmatrix}
\right|^2 }{\begin{matrix}K^{n+1}\det\left(\frac{\partial^{2}}{\partial z_{i}\partial\bar z_{j}} \log K\right)\end{matrix}}\end{array}\right|_{z=z_{s_{k}}}=\left.\begin{array}{c} \frac{\left|\begin{matrix}\det\end{matrix}\begin{pmatrix}
                                                     \tilde f_{0}(z)&\dots & \tilde f_{n}(z)\\
\frac{ \partial \tilde f_{0}}{\partial z_{1}}(z)&\dots & \frac{ \partial \tilde f_{n}}{\partial z_{1}}(z)\\
\vdots&\ddots &\vdots\\
\frac{ \partial \tilde f_{0}}{\partial z_{n}}(z)&\dots &\frac{ \partial \tilde f_{n}}{\partial z_{n}}(z)\\
                                                    \end{pmatrix}
\right|^2 }{\begin{matrix}K^{n+1}\det\left(\frac{\partial^{2}}{\partial z_{i}\partial\bar z_{j}} \log K\right)\end{matrix}}\end{array}\right|_{z=z_{s_{k}}}$$
$$\leq \det\begin{pmatrix} \langle \tilde f_{0},\tilde f_{0}\rangle_{\lh} &\cdots&\langle \tilde f_{n},\tilde f_{0}\rangle_{\lh}\\
 \vdots&\ddots&\vdots\\
  \langle \tilde f_{0},\tilde f_{n}\rangle_{\lh}&\cdots&\langle \tilde f_{n},\tilde f_{n}\rangle_{\lh} \end{pmatrix} \leq \Vert \tilde f_{0}\Vert_{\lh}^{2}\dots\Vert \tilde f_{0}\Vert_{\lh}^{2}$$$$\leq C^{n+1} \int_{\{G_{\Omega}(\cdot,z_{s_{k}})<-1\}}|f_{0}|^2 d\lambda\dots\int_{\{G_{\Omega}(\cdot,z_{s_{k}})<-1\}}|f_{n}|^2d \lambda\to 0,$$
because each $f_{j}\in\lh$ and the volume of $\{G_{\Omega}(\cdot,z)<-1\}$ goes to $0$ as $k\to\infty$ in bounded hyperconvex domains, see \cite{MR1650305} or \cite{MR1714284}.

\end{proof}

We do not know whether or not there exist domains for which one of the metrics $T_{i\bar j}$, $\tilde T_{i\bar j}$ is complete and the other is not.

Despite the suggestion in \cite{MR0112162} that domains in which the Bergman metric is complete should satisfy the conditions in Theorem \ref{koba}, Zwonek in \cite{MR1869096} constructed a domain for which the Bergman metric is complete and the limit in  (\ref{weakkob}) is not zero. This means that the Kobayashi criterion is not a if and only if statement. As noted  in \cite{MR2139520} it is not known whether or not the modified version Theorem \ref{blo} is a if and only if statement.
Likewise we do not know whether or not the criterion in Theorem \ref{main} is a if and only if statement. We do not know this even if the limit in (\ref{inmain}) is assumed to be $0$. 

\end{section}
\bibliographystyle{amsplain.bst}
\def\cprime{$'$}
\providecommand{\bysame}{\leavevmode\hbox to3em{\hrulefill}\thinspace}
\providecommand{\MR}{\relax\ifhmode\unskip\space\fi MR }
\providecommand{\MRhref}[2]{%
  \href{http://www.ams.org/mathscinet-getitem?mr=#1}{#2}
}
\providecommand{\href}[2]{#2}

\end{document}